\documentclass[12pt, letterpaper]{amsart}

\newcommand{\indentalign}{\hspace{0.3in}&\hspace{-0.3in}}

\renewcommand{\Re}{\operatorname{Re}}
\renewcommand{\Im}{\operatorname{Im}}

\newcommand{\defeq}{\stackrel{\rm{def}}{=}}

\newcommand{\sgn}{\operatorname{sgn}}

\newcommand{\bear}{\begin{eqnarray}}
\newcommand{\eear}{\end{eqnarray}}
\newcommand{\beeq}{\begin{equation}}
\newcommand{\eneq}{\end{equation}}

\def\bm{\begin{bmatrix}}
\def\endm{\end{bmatrix}}

\def\de{\delta}

\let\phi\varphi

\usepackage{hyperref}
\hypersetup{
colorlinks=true,
linkcolor=black,
citecolor=black,
urlcolor=blue,
linktoc=page,
pdfstartview=FitH
}

\allowdisplaybreaks

\usepackage{amssymb}
\usepackage{amsthm}
\usepackage{amsxtra}
\usepackage{graphicx}

\newtheorem{theorem}{Theorem}

\newtheorem{proposition}[theorem]{Proposition}
\newtheorem{lemma}[theorem]{Lemma}
\newtheorem{corollary}[theorem]{Corollary}

\theoremstyle{remark}

\numberwithin{equation}{section}

\numberwithin{theorem}{section}

\numberwithin{table}{section}

\numberwithin{figure}{section}

\ifx\pdfoutput\undefined
  \DeclareGraphicsExtensions{.pstex, .eps}
\else
  \ifx\pdfoutput\relax
    \DeclareGraphicsExtensions{.pstex, .eps}
  \else
    \ifnum\pdfoutput>0
      \DeclareGraphicsExtensions{.pdf}
    \else
      \DeclareGraphicsExtensions{.pstex, .eps}
    \fi
  \fi
\fi

\title[Blow-up for 1D NLS with point nonlinearity I]{Blow-up for the 1D nonlinear Schr\"odinger equation with point nonlinearity I: Basic theory}

\author{Justin Holmer}
\author{Chang Liu}
\address{Brown University}

\begin{document}

\maketitle

\begin{abstract}
We consider the 1D nonlinear Schr\"odinger equation (NLS) with focusing \emph{point nonlinearity},
\begin{equation}
\label{E:intro-1}
i\partial_t\psi + \partial_x^2\psi + \delta|\psi|^{p-1}\psi = 0,
\end{equation}
where $\de=\de(x)$ is the delta function supported at the origin.    In this work, we show that \eqref{E:intro-1} shares many properties in common with those previously established for the focusing autonomous translationally-invariant NLS 
\begin{equation}
\label{E:intro-2}
i\partial_t \psi + \Delta \psi + |\psi|^{p-1}\psi=0 \,.
\end{equation}
The critical Sobolev space $\dot H^{\sigma_c}$ for \eqref{E:intro-1} is $\sigma_c=\frac12-\frac{1}{p-1}$, whereas for \eqref{E:intro-2} it is $\sigma_c=\frac{d}{2}-\frac{2}{p-1}$.  In particular, the $L^2$ critical case for \eqref{E:intro-1} is $p=3$.  
We prove several results pertaining to blow-up for \eqref{E:intro-1} that correspond to key classical results for \eqref{E:intro-2}.  Specifically, we (1) obtain a sharp Gagliardo-Nirenberg inequality analogous to Weinstein \cite{MR691044}, (2) apply the sharp Gagliardo-Nirenberg inequality and a local virial identity to obtain a sharp global existence/blow-up threshold analogous to Weinstein \cite{MR691044}, Glassey \cite{MR0460850} in the case $\sigma_c=0$ and Duyckaerts, Holmer, \& Roudenko \cite{MR2470397}, Guevara \cite{MR3266698}, and Fang, Xie, \& Cazenave, \cite{MR2838120} for $0<\sigma_c<1$,  (3) prove a sharp mass concentration result in the $L^2$ critical case analogous to Tsutsumi \cite{MR1074950}, Merle \& Tsutsumi \cite{MR1047566} and (4) show that minimal mass blow-up solutions in the $L^2$ critical case are pseudoconformal transformations of the ground state, analogous to Merle \cite{merle1993determination}.

\end{abstract}

\section{Introduction}

We consider the 1D nonlinear Schr\"odinger equation (NLS) with $p$-power focusing \emph{point nonlinearity}, for $p>1$,
\begin{equation}
\label{E:1-120}
i\partial_t \psi + \partial_x^2 \psi + \delta |\psi|^{p-1}\psi =0
\end{equation}
where $\de=\de(x)$ is the delta function supported at the origin, and $\psi(x,t)$ is a complex-valued wave function for $x\in \mathbb{R}$.  The equation \eqref{E:1-120} can be interpreted as the free linear Schr\"odinger equation
$$i\partial_t \psi + \partial_x^2 \psi =0 \,, \quad \text{for} \quad x\neq 0$$
together with the jump conditions at $x=0$:\footnote{We define $\psi(0-,t) = \lim_{x\nearrow 0} \psi(x,t)$ and $\psi(0+,t) = \lim_{x\searrow 0} \psi(x,t)$.}  
\begin{equation}
\label{E:jump}
\begin{aligned}
&\psi(0,t) \defeq \psi(0-,t) = \psi(0+,t) \\
&\partial_x \psi(0+,t) - \partial_x\psi (0-,t) = - |\psi(0,t)|^{p-1}\psi(0,t)
\end{aligned}
\end{equation}
Appropriate function spaces in which to discuss solutions, and the corresponding meaning of \eqref{E:jump} will be given below.  Specifically important is the case $p=3$, the 1D focusing NLS with \emph{cubic} point nonlinearity
\begin{equation}
\label{E:1dcc}
i\partial_t\psi + \partial_x^2\psi + \delta |\psi|^2\psi =0
\end{equation}

This equation could model the following two physical settings, as proposed by \cite{MolinaBustamante}.  First, \eqref{E:1dcc} could model an electron propagating in a 1D linear medium which contains a vibrational ``impurity'' at the origin that can couple strongly to the electron.  In the approximation where one considers the vibrations completely ``enslaved'' to the electron,  \eqref{E:1dcc} is obtained as the effective equation for the electron.  \cite{MolinaBustamante} further remark that an important application of \eqref{E:1dcc} is that of a wave propagating in a 1D linear medium which contains a narrow strip of nonlinear (general Kerr-type) material. This nonlinear strip is assumed to be much smaller than the typical wavelength.  In fact, periodic and quasiperiodic arrays of nonlinear strips have been considered by a straightforward generalization of \eqref{E:1dcc} in order to model wave propagation in some nonlinear superlattices \cite{Hennig}.

We will begin our mathematical discussion of \eqref{E:1-120} with a comparison to the $d$D ($x\in \mathbb{R}^d$) pure-power autonomous translationally invariant focusing NLS
\begin{equation}
\label{E:1-121}
i\partial_t \psi + \Delta \psi + |\psi|^{p-1}\psi=0
\end{equation}
The equation \eqref{E:1-121} is a well-studied example of a scalar nonlinear dispersive wave equation -- see the textbook \cite{MR2233925} for an overview of this field.  The equation \eqref{E:1-121} obeys the scaling law
$$\psi(x,t) \text{ solves }\eqref{E:1-121} \implies \lambda^{2/(p-1)}\psi(\lambda x,\lambda^2t) \text{ solves }\eqref{E:1-121}$$
The scale-invariant Sobolev norm $\dot H^{\sigma_c}$ (satisfying $\|\psi\|_{\dot H^{\sigma_c}} = \| \psi_\lambda \|_{\dot H^{\sigma_c}}$) is
$$\sigma_c = \frac{d}{2} - \frac{2}{p-1} \qquad \text{(for \eqref{E:1-121})}$$
A particularly important case is $\sigma_c=0$, the $L^2$-critical, or mass-critical, case $p= 1+ \frac{4}{d}$.  Note that for $d=1$, the $L^2$ critical nonlinearity is $p=5$ (quintic)
and for $d=2$, the $L^2$ critical case is the cubic nonlinearity $p=3$.  On the other hand, the equation \eqref{E:1-120} obeys the scaling law
\begin{equation}
\label{E:1-144}
\psi(x,t) \text{ solves }\eqref{E:1-120} \implies \lambda^{1/(p-1)}\psi(\lambda x,\lambda^2t) \text{ solves }\eqref{E:1-120}
\end{equation}
The scale-invariant Sobolev norm $\dot H^{\sigma_c}$ (satisfying $\|\psi\|_{\dot H^{\sigma_c}} = \| \psi_\lambda \|_{\dot H^{\sigma_c}}$) is
$$\sigma_c = \frac{1}{2} - \frac{1}{p-1} \qquad \text{(for \eqref{E:1-120})} $$
Thus, for \eqref{E:1-120}, the point  nonlinearity, $p=3$ (cubic) is the $L^2$ critical setting.  

The standard equation \eqref{E:1-121} satisfies mass, energy, and  momentum conservation laws (meaning quantities independent of $t$)
$$M(\psi) = \| \psi\|_{L^2}^2$$
$$E(\psi) = \frac12 \|\nabla \psi \|_{L^2}^2 - \frac{1}{p+1} \|\psi\|_{L^{p+1}}^{p+1}$$
$$P(\psi) = \Im \int \bar \psi \; \nabla \psi $$
For \eqref{E:1-120}, the conservation laws take the form
$$M(\psi) = \|\psi\|_{L^2}^2$$
$$E(\psi) = \frac12 \|\partial_x \psi\|_{L^2}^2 - \frac{1}{p+1} |\psi(0, \cdot)|^{p+1}$$
There is no conservation of momentum, since Noether's theorem links this to translational invariance, a property that \eqref{E:1-121} has but \eqref{E:1-120} does not.  In \S\ref{S:conserved}, we prove these conservation laws, and also state and prove and important identity, the local virial identity, which has applications to the blow-up results in Theorem \ref{T:global-blowup}(2), \ref{T:global-blowup-critical}(2) below.

In this paper, we will show that many classical results for \eqref{E:1-121} have counterparts for \eqref{E:1-120}.  In some cases, the results for \eqref{E:1-120} are stronger or at least simpler to prove.  

Local well-posedness for \eqref{E:1-121} in $H_x^s$ was obtained by Ginibre \& Velo \cite{MR0533218}, Kato \cite{MR0877998}, and Cazenave-Weissler \cite{MR1055532}.  For \eqref{E:1-120}, an $H_x^1$ (energy space) local well-posedness theory was given by Adami \& Teta \cite{MR1814425}.  Alternatively, one can use the `time traces' estimates in Holmer \cite{holmer2005initial} or apply the method of Bona, Sun, \& Zhang \cite{BonaSunZhang} as recently done by Batal \& \"Ozsari \cite{Batal}.  (See also Erdogan \& Tzirakis \cite{ET15} for recent work on the half-line problem).   Specifically, we have

\begin{theorem}[$H_x^1$ local well-posedness]
\label{T:local}
Given initial data $\psi_0 \in H_x^1$, there exists $T = T(\|\psi_0\|_{H^1})$ and a solution $\psi(x,t)$ to \eqref{E:1-120} on $[0,T)$ satisfying $\psi(0) = \psi_0$ and
\begin{align*}
& \psi \in C([0,T]; H_x^1) \\
& \psi \in L_{[0,T]}^qW_x^{1,p}, \text{ for Strichartz admissible pairs } \frac{2}{q} + \frac{1}{p} = \frac12 \\
& \psi \in C( x\in \mathbb{R}; H_{[0,T]}^{3/4}) \\
& \partial_x \psi \in C( x\in \mathbb{R}\backslash \{0\}; H_{[0,T]}^{1/4}) \\
& \partial_x \psi(0-,t) \defeq \lim_{x\nearrow 0} \partial_x \psi(x,t) \text{ and } \partial_x \psi(0+,t) \defeq \lim_{x\searrow 0}\partial_x\psi(x,t) \text{ exist in }H_{[0,T]}^{1/4}\\
& \partial_x \psi (0+,t) - \partial_x \psi(0-,t) = - |\psi(0,t)|^{p-1}\psi(0,t)  \text{ as an equality of }H_{[0,T]}^{1/4} \text{ functions}
\end{align*}
Among all solutions satisfying the above regularity conditions, it is unique.  Moreover, the data-to-solution map $\psi_0 \mapsto \psi$, as a map $H_x^1 \to C([0,T]; H_x^1)$, is continuous.
\end{theorem}
It is important that the limits $\partial_x \psi(0-,t) \defeq \lim_{x\nearrow 0} \partial_x \psi(x,t)$  and $\partial_x \psi(0+,t) \defeq \lim_{x\searrow 0}\partial_x\psi(x,t)$ do not exist pointwise in $t$, but in the Sobolev space $H_t^{1/4}$.  One particular consequence of the local theory is that if the maximal forward time $T_*>0$ of existence is finite (i.e. $T_*<\infty$) then necessarily
$$\lim_{t\nearrow T_*} \| \psi_x(t) \|_{L^2} = +\infty$$
In this case, we say that the solution $\psi(t)$ \emph{blows-up} at time $t=T_*>0$.  In fact, the rate of divergence of this norm has a lower bound:

\begin{proposition}[lower bound on the blow-up rate]
\label{P:lower-bound}
Suppose that $\psi(t)$ is an $H_x^1$ solution on $[0,T_*)$ that blows-up at time $t=T_*<\infty$.  Then there exists an absolute constant\footnote{$C$ is independent of the initial condition, and only depends on $p$} $C>0$ such that if $\| \psi_x(t) \|_{L_x^2} \geq |E(\psi)|^{1/2}$, then
$$\| \psi_x(t) \|_{L^2} \geq C(T_*-t)^{-(1-\sigma_c)/2}$$
\end{proposition}

Note that by definition of blow-up, $\lim_{t\nearrow T_*} \|\psi_x(t)\|_{L^2} = \infty$, and hence there is some time interval $(T_*-\delta,T_*)$ on which the condition $\| \psi_x(t) \|_{L_x^2} \geq |E(\psi)|^{1/2}$ applies.  Theorem \ref{T:local} is closely related to the half-line result in \cite{Batal}, and the proof given there can be adapted to prove Theorem \ref{T:local}.  Prop. \ref{P:lower-bound} follows from the local theory estimates used to prove Theorem \ref{T:local} and energy conservation.  The result for \eqref{E:1-121}, analogous to Prop. \ref{P:lower-bound}, is given by Cazenave \& Weissler \cite{MR1055532}.  The proofs of Theorem \ref{T:local} and Prop. \ref{P:lower-bound} will not be included in this paper.

Both \eqref{E:1-120} and \eqref{E:1-121} have solitary wave solutions $\psi(x,t) = e^{it} \varphi_0(x)$, where $\varphi_0$ solves the stationary equation
\begin{align}
\label{E:1-123} & \text{for } \eqref{E:1-120}, && 0=\varphi_0 - \partial_x^2 \varphi_0 - \delta |\varphi_0|^{p-1}\varphi_0 \\
\label{E:1-124} & \text{for } \eqref{E:1-121}, && 0= \varphi_0 - \Delta \varphi_0 - |\varphi_0|^{p-1} \varphi_0 
\end{align}
In the case of \eqref{E:1-124}, there exist a unique (up to translation) radial, smooth, exponentially decaying solution called the \emph{ground state}.  Existence via concentration compactness was obtained by Berestycki \& Lions \cite{MR695535}, and uniqueness was proved by Kwong \cite{MR969899}.  Weinstein \cite{MR691044} proved that this ground state $\varphi_0$ saturates the Gagliardo-Nirenberg inequality
\begin{equation}
\label{E:sharpGN}
\| \psi \|_{L^{p+1}} \leq c_{GN} \| \psi \|_{L^2}^{1-\sigma_1} \| \nabla \psi \|_{L^2}^{\sigma_1}
\end{equation}
where
$$\sigma_1 = \frac{d}{2} \frac{p-1}{p+1} \text{ equivalently } \frac{1}{p+1} = \frac12 - \frac{\sigma_1}{d}$$
and
$$c_{GN} = \frac{ \| \varphi_0 \|_{L^{p+1}}}{ \| \varphi_0 \|_{L^2}^{1-\sigma_1} \| \nabla \varphi_0 \|_{L^2}^{\sigma_1}}$$
The Pohozhaev identities are
$$
 \sigma_1 \| \varphi_0 \|_{L^2}^2 = (1-\sigma_1) \| \nabla \varphi_0 \|_{L^2}^2 = (1-\sigma_1)\sigma_1 \| \varphi_0 \|_{L^{p+1}}^{p+1}
$$
The corresponding statements for \eqref{E:1-123} are more straightforward.  For $x\neq 0$, the equation is $0 = \varphi_0 - \partial_x^2 \varphi_0$, which is linear with solution space spanned by $e^{\pm x}$.  Thus any $L^2$ solution to \eqref{E:1-123} must be of the form 
$$\varphi_0(x) = \begin{cases} \alpha e^{-x} & \text{for } x>0 \\ \beta e^x & \text{for }x<0 \end{cases}$$
Continuity across $x=0$ forces $\alpha=\beta$ and the jump in derivative condition forces $\alpha = 2^{1/(p-1)}$.  Consequently
\begin{equation}
\label{E:1-125}
\varphi_0(x) = 2^{1/(p-1)} e^{-|x|}
\end{equation}
is the only solution to \eqref{E:1-123}. The Pohozhaev identities take the form
$$\| \varphi_0 \|_{L^2}^2 = \|\partial_x \varphi_0 \|_{L^2}^2 = \frac12 \|\varphi_0\|_{L^\infty}^{p+1} = 2^{\frac{2}{p-1}}$$
and can be verified by direct computation from \eqref{E:1-125}.  Moreover, the analogue of Weinstein \cite{MR691044} is the following.

\begin{proposition}[sharp Gagliardo-Nirenberg inequality]
\label{P:sharpGN}
For any $\psi \in H^1$, 
\begin{equation}
\label{E:1-126}
|\psi(0)|^2 \leq \| \psi \|_{L^2} \|\psi' \|_{L^2} \,.
\end{equation} Equality is achieved if and only if there exists $\theta\in \mathbb{R}$, $\alpha>0$, and $\beta>0$  such that  $\psi(x)=e^{i\theta} \alpha \phi_0(\beta x)$.
\end{proposition}

This is proved in \S \ref{S:sharpGN}.

For \eqref{E:1-121}, sharp threshold conditions for global well-posedness and for blow-up are available in the $L^2$ supercritical case $p>1+\frac{4}{d}$.  These are given in Duyckaerts, Holmer, \& Roudenko \cite{MR2470397} for the case $p=3$, $d=3$, and for general $L^2$ supercritical, energy subcritical NLS, in Guevara \cite{MR3266698} and Fang, Xie, \& Cazenave, \cite{MR2838120}.  The analogous result for \eqref{E:1-120} is

\begin{theorem}[$L^2$ supercritical global existence/blow-up dichotomy]
\label{T:global-blowup}
Suppose that $\psi(t)$ is an $H_x^1$ solution of \eqref{E:1-120} for $p>3$ satisfying
$$M(\psi)^{\frac{1-\sigma_c}{\sigma_c}} E(\psi)< M(\varphi_0)^{\frac{1-\sigma_c}{\sigma_c}} E(\varphi_0)$$
Let
$$\eta(t) = \frac{ \|\psi\|_{L^2}^{\frac{1-\sigma_c}{\sigma_c}} \|\psi_x(t) \|_{L_x^2}}{\| \varphi_0 \|_{L^2}^{\frac{1-\sigma_c}{\sigma_c}} \|(\varphi_0)_x \|_{L^2}}$$
Then
\begin{enumerate}
\item If $\eta(0)<1$, then the solution $\psi(t)$ is global in both time directions and $\eta(t)<1$ for all $t\in \mathbb{R}$.
\item If $\eta(0)>1$, then the solution $\psi(t)$ blows-up in the negative time direction at some $T_-<0$, blows-up in the positive time direction at some $T_+>0$, and $\eta(t)>1$ for all $t\in (T_-,T_+)$.   
\end{enumerate}
\end{theorem}

This is proved in \S \ref{S:sharpGN}.  We remark that Case (2) of Theorem \ref{T:global-blowup} does not require the assumption of finite variance, in contrast to the known results for \eqref{E:1-121}.

Next we concentrate on the $L^2$ critical case $p=3$.  In the case of \eqref{E:1-121}, Weinstein \cite{MR691044} proved the sharp threshold condition for global existence, and Glassey \cite{MR0460850} proved that negative energy solutions of finite variance blow-up.  The analogous result for \eqref{E:1dcc} using \eqref{E:1-126} and the virial identity is the following.  We note that, in contrast to the known results for \eqref{E:1-121}, for \eqref{E:1dcc} we do not need the assumption of finite variance for the blow-up result. 

\begin{theorem}[$L^2$ critical global existence/blow-up dichotomy] 
\label{T:global-blowup-critical}
Suppose that $\psi(t)$ is an $H_x^1$ solution to \eqref{E:1dcc}.
\begin{enumerate}
\item If $M(\psi)<M(\phi_0)=2$, then $E(\psi)>0$ and $\psi(t)$ satisfies the bound
$$\|\psi_x(t)\|_{L^2}^2 \leq  \frac{2E(\psi)}{1- \frac12M(\psi)}$$
and is hence a global solution (no blow-up).  
\item If $E(\psi)<0$ then $\psi(t)$ blows-up in finite time.
\end{enumerate}
\end{theorem}

This is proved in \S \ref{S:sharpGN}.  The global result (1) is sharp since (as we remark below) there exist finite time blow-up solutions $\psi$ with $M(\psi)=M(\varphi_0)$.  The blow-up result (2) is sharp since the solution $\psi(x,t) = e^{it}\varphi_0(x)$ is global and $E(\psi)=E(\varphi_0)=0$.   While all solutions of negative energy blow-up, there do exist positive energy blow-up solutions.

In the case $L^2$ critical case $p=3$, there is an additional symmetry, \emph{pseudoconformal transform}, which is 
$$\psi(x,t) \text{ solves \eqref{E:1dcc}} \qquad \mapsto \qquad \tilde \psi(x,t) = \frac{e^{ix^2/4t}}{t^{1/2}}\psi( \frac{x}{t}, - \frac{1}{t}) \text{ solves \eqref{E:1dcc}.}$$
This transformation is also valid for \eqref{E:1-121} in the $L^2$ critical case $p=1+\frac{4}{d}$, as was observed in that context by Ginibre \& Velo \cite{MR0533218}.  Applying the pseudoconformal transformation to the ground state solution $e^{it}\varphi_0(x)$, together with time reversal, time translation, and scaling symmetries, gives the solution $S_{\lambda,T_*}(x,t)$ to \eqref{E:1dcc}, for any $T_*>0$ and $\lambda>0$, where
\begin{equation}
\label{E:1-148}
S_{\lambda,T_*}(x,t) = \frac{e^{i/\lambda^2(T_*-t)} e^{-ix^2/4(T_*-t)}}{[\lambda(T_*-t)]^{1/2}} \varphi_0\left( \frac{x}{\lambda(T_*-t)} \right) \qquad \text{for }t<T_*
\end{equation}
This is a solution with initial condition
\begin{equation}
\label{E:1-149}
S_{\lambda,T_*}(x,0) = \frac{ e^{i/\lambda^2T_*} e^{-ix^2/4T_*}}{(\lambda T_*)^{1/2}} \varphi_0\left( \frac{x}{\lambda T_*} \right)
\end{equation}
that blows-up at forward time $t=T_*>0$.  Moreover, $\| S_{\lambda,T_*} \|_{L^2_x} = \|\varphi_0\|_{L^2}$, and hence by Theorem \ref{T:global-blowup-critical}, $S_{\lambda,T_*}$ is a \emph{minimal mass} blow-up solution to \eqref{E:1dcc}.  Remarks on the proof of the pseudoconformal transformation and the derivation of $S_{\lambda,T_*}(x,t)$ are included below in \S \ref{S:conserved}.

Tsutsumi \cite{MR1074950}  and Merle \& Tsutsumi \cite{MR1047566} showed that blow-up solutions to the $L^2$ critical case of \eqref{E:1-121} concentrate at least the ground state mass at each blow-up point.  The corresponding result for \eqref{E:1dcc} is

\begin{theorem}[Mass concentration of $L^2$-critical blow-up solutions]
\label{T:mass-conc}
Suppose that $\psi(t)$ is any $H^1$ solution to \eqref{E:1dcc} that blows-up at some forward time $T_*>0$.  Let $\mu(t)$ be any function so that $\lim_{t\nearrow T_*} \mu(t) = +\infty$.  Then
$$\liminf_{t\nearrow T_*} \int_{|x| \leq \mu(t) \|\psi_x(t)\|_{L^2}^{-1}} |\psi(x,t)|^2 \, dx \geq \|\varphi_0\|_{L^2}^2$$
\end{theorem}

This is proved in \S \ref{S:L2}.

Following Merle \cite{merle1993determination}, Hmidi-Keraani \cite{hmidi2005blowup} we can show that a minimal mass blow-up solution of \eqref{E:1dcc} is necessarily a pseudoconformal transformation of the ground state.

\begin{theorem}[characterization of minimal mass $L^2$ critical blow-up solutions]
\label{T:minimal}
Suppose that $\psi(t)$ is an $H^1$ solution to \eqref{E:1dcc} that blows-up at time $T_*>0$ and $\|\psi_0\|_{L^2}=\|\varphi_0\|_{L^2}$.  Then there exists $\theta\in \mathbb{R}$ and $\beta>0$ such that 
\begin{equation}
\label{E:1-147}
\psi_0(x) = e^{i\theta} e^{-i|x|^2/(4T_*)} \beta^{1/2}  \varphi_0(\beta x)
\end{equation}

\end{theorem}

This is proved in \S \ref{S:L2}.  Note that \eqref{E:1-147} matches \eqref{E:1-149} with $\beta = 1/(\lambda T_*)$, so the solution $\psi(t)$ is given by $S_{\lambda,T_*}(x,t)$ expressed in \eqref{E:1-148}, up to a phase factor.

Finally, we consider the behavior of \emph{near minimal mass} blow-up solutions $\psi$ in the $L^2$ critical case.  By this we mean solutions $\psi$ for which
\begin{equation}
\label{E:near-minimal}
2= M(\varphi_0) < M(\psi) < M(\varphi_0)+\delta = 2+\delta
\end{equation}
and for which there exists $T_*>0$ such that $\psi(t)$ blows-up at $t=T_*$.  (Blow-up will necessarily occur if $E(\psi)<0$ by Theorem \ref{T:global-blowup-critical}(2), and can possibly occur if $E(\psi)\geq 0$)

\begin{theorem}[$L^2$ critical near minimal mass blow-up solutions]
\label{T:near-minimal}
For each $\epsilon>0$, there exists $\delta>0$ such that if $\psi(t)$ satisfies \eqref{E:near-minimal} and $\psi(t)$ blows-up at time $t=T_*>0$, then for all $t$ sufficiently close to $T_*$ there exist $\theta(t)\in \mathbb{R}$ and $\rho(t)>0$ such that
$$\| e^{-i\theta(t)} \rho(t)^{1/2}\psi(\rho(t)x,t) - \varphi_0(x) \|_{H_x^1} \leq \epsilon$$
The parameters $\rho(t)$ and $\theta(t)$ can be chosen to be continuous.
\end{theorem}

This is proved in \S \ref{S:near-minimal}.  It states that near minimal mass blow-up solutions are, up to scaling and phase, perturbations of the ground state.  In Paper III-IV, we will construct two families of near minimal mass blow-up solutions for \eqref{E:1dcc}.  In Paper III, we will construct the Bourgain \& Wang \cite{MR1655515} blow-up solutions, which have positive energy and for which $\|\psi_x(t)\|_{L_x^2} \sim (T_*-t)^{-1}$, and in Paper IV, we will construct the log-log blow-up solutions \cite{MR807329, MR966356, perelman2001formation, MR1995801, MR2061329, MR2150386, MR2116733, MR2169042, MR2122541} that have negative energy and for which $\|\psi_x(t)\|_{L_x^2} \sim (T_*-t)^{-1/2}(\log\log(T_*-t)^{-1})^{1/2}$.  In Paper II, we will give single- and multi-bump blow-up profiles \cite{MR1349311, MR1699715, MR1975789, MR2729284} in the slightly $L^2$ supercritical regime, $3<p<3+\delta$.

Textbook treatments of blow-up for NLS-type equations are Sulem \& Sulem \cite{MR1696311} and Fibich \cite{MR3308230}.  To our knowledge, blow-up for concentrated point nonlinearities has only been studied in the 3D case in Adami, Dell'Antonio, Figari, \& Teta \cite{MR2037249} (see also the announcement Adami \cite{MR1946010}), obtaining some blow-up criteria from the virial identity and observing the explicit blow-up solution obtained in the $L^2$-critical case by applying the pseudoconformal transformation to the solitary wave.

\subsection*{Acknowledgments}

We thank Catherine Sulem, Galina Perelman, and Maciej Zworski for discussions about this topic, suggestions, and encouragement.  The material in this paper will be included as part of the PhD thesis of the second author at Brown University.  While this work was completed, the first author was supported in part by NSF grants DMS-1200455,  DMS-1500106.  The second author was supported in part by NSF grant DMS-1200455 (PI Justin Holmer).

\section{Conserved quantities and transformations}
\label{S:conserved}

To derive the energy conservation law (and other conservation laws), one assumes that the solution has higher regularity, and then the case of a general solution belonging to $C([0,T]; H_x^1)$ follows by an approximation argument.  Specifically, one assumes that $\psi(x,t)$, $\psi_x(x,t)$, $\psi_{xx}(x,t)$, and $\psi_t(x,t)$ are continuous for $x\neq 0$, all $t$, and each have left- and right-hand limits at $x=0$.  Moreover, the following juncture conditions hold
\begin{equation}
\label{E:jump2}
\begin{aligned}
& \psi(0) \defeq \psi(0-)= \psi(0+) \\
& \psi_x(0+)-\psi_x(0-) = - |\psi(0)|^{p-1}\psi(0) \\
& \psi_{xx}(0) \defeq \psi_{xx}(0-)=\psi_{xx}(0+) \\
& \psi_t(0) \defeq \psi_t(0-)=\psi_t(0+)
\end{aligned}
\end{equation}
Thus, $\psi$ ,$\psi_{xx}$, $\psi_t$ are assumed continuous across $x=0$, but note that due to the jump in $\psi_x$ across $x=0$, it follows that $\psi_{xx}$, when computed in the distributional sense, is equal to a continuous function (that we denote $\psi_{xx}$) \emph{plus} the distribution $- \delta(x) |\psi(0)|^{p-1}\psi(0)$.  One can now use \eqref{E:jump2} plus the fact that $i\psi_t + \psi_{xx}=0$ for $x\neq 0$ to derive the energy conservation.  Indeed, one pairs the equation with $\bar \psi_t$, integrates $x$ over $(-\infty,+\infty)$ and takes the real part to obtain
$$0= \Re \int_{-\infty}^{+\infty} \psi_{xx} \bar \psi_t \, dx $$
When integrating by parts, however, one must remember the jump in $\psi_x$ and that yields boundary terms at $x=0$
\begin{align*}
0 &=- \Re \int_{-\infty}^{+\infty} \psi_x \bar \psi_{xt} \, dx - \Re \psi_x(0+) \bar \psi_t(0+) + \Re \psi_x(0-) \bar \psi_t(0-) \\
&=- \Re \int_{-\infty}^{+\infty} \psi_x \bar \psi_{xt} \, dx - \Re (\psi_x(0+)-\psi_x(0-) \bar \psi_t(0) \\
&= - \Re \int_{-\infty}^{+\infty} \psi_x \bar \psi_{xt} \, dx + \Re |\psi(0)|^{p-1} \psi(0) \bar \psi_t(0) \\
& = \frac{d}{dt} \left( - \frac12 \int |\psi_x|^2 \, dx + \frac{1}{p+1} |\psi(0)|^{p+1} \right)
\end{align*}

The standard NLS \eqref{E:1-121} satisfies the virial identity
$$\partial_t^2 \int |x|^2 |\psi|^2 \, dx = 8 \| \nabla \psi \|_{L^2}^2 - \frac{4d(p-1)}{p+1} \| \psi \|_{L^{p+1}}^{p+1}$$
$$ = 4d(p-1)E(\psi) + (8-2d(p-1))\|\nabla \psi \|_{L^2}^2$$
In the $L^2$ critical case $p=1+\frac{4}{d}$, this reduces to just $16E(\psi)$.  

In the case of the point nonlinearity \eqref{E:1-120}, the virial identity is
\begin{align*}
\partial_t^2 \int |x|^2 |\psi|^2 \, dx &= 8 \| \psi_x\|_{L^2}^2  - 4|\psi(0)|^{p+1} \\
&= 4(p+1)E(\psi) + (8-2(p+1))\|\psi_x\|_{L^2}^2
\end{align*}
In fact, we have a generalization, the local virial identity, which is the following.  For weight function $a(x)$ satisfying  $a(0)=a_x(0)=a_{xxx}(0)=0$, solution $\psi$ to \eqref{E:1-120} satisfy
\begin{equation}
\label{E:local-virial}
\partial_t^2 \int a(x) |\psi|^2 \, dx = 4 \int a_{xx} |\psi_x|^2 - 2 a_{xx}(0)|\psi(0)|^{p+1}-\int a_{xxxx} |\psi|^2
\end{equation}

We will now prove \eqref{E:local-virial} for $\psi$ satisfying the jump conditions \eqref{E:jump2}.  Since for $x\neq 0$, $\psi_t = i \psi_{xx}$, we have
$$\partial_t \int_{-\infty}^{+\infty} a|\psi|^2 = 2 \Re \int_{-\infty}^{+\infty} a \bar \psi \psi_t = -2 \Im \int_{-\infty}^{+\infty} a \bar \psi \psi_{xx}$$ 
Integration by parts gives no boundary terms since $a(0)=0$, 
$$\partial_t \int_{-\infty}^{+\infty} a|\psi|^2= 2 \Im \int_{-\infty}^{+\infty} a_x \bar \psi \psi_x$$
Taking another time derivative gives
$$\partial_t^2 \int_{-\infty}^{+\infty} a |\psi|^2 = 2 \Im \int_{-\infty}^{+\infty} a_x \bar \psi_t \psi_x + 2\Im \int_{-\infty}^{+\infty} a_x \bar \psi \psi_{xt}$$
In the second term, we integrate by parts in $x$, which again leaves no boundary terms since $a_x(0)=0$,
$$\partial_t^2 \int_{-\infty}^{+\infty} a |\psi|^2= -4 \Im \int_{-\infty}^{+\infty} a_x \bar \psi_x \psi_t - 2\Im \int_{-\infty}^{+\infty} a_{xx} \bar \psi \psi_t$$
Substituting $\psi_t = i \psi_{xx}$,
\begin{equation}
\label{E:1-156}
\partial_t^2 \int_{-\infty}^{+\infty} a |\psi|^2= - 4\Re \int_{-\infty}^{+\infty} a_x \bar \psi_x \psi_{xx} - 2 \Re \int_{-\infty}^{+\infty} a_{xx} \bar \psi \psi_{xx} = \text{I} + \text{II}
\end{equation}
For $\text{I}$, we find
$$\text{I} = -2 \int_{-\infty}^{+\infty} a_x (|\psi_x|^2)_x$$
Integrating by parts in $x$ leaves no boundary terms since $a_x(0)=0$,
\begin{equation}
\label{E:1-157}
\text{I} = 2 \int_{-\infty}^{+\infty} a_{xx} |\psi_x|^2 
\end{equation}
For term $\text{II}$, integration by parts does leave boundary terms since $a_{xx}(0)\neq 0$
$$\text{II} = 2 \Re \int_{-\infty}^{+\infty} (a_{xx}\bar\psi)_x \psi_x + 2\Re a_{xx}(0) \bar \psi(0) (\psi_x(0+)-\psi_x(0-))$$
Substituting the boundary condition
\begin{align*}
\text{II} &= 2 \Re \int_{-\infty}^{+\infty} a_{xxx} \bar \psi \psi_x + 2 \int_{-\infty}^{+\infty} a_{xx} |\psi_x|^2 - 2 a_{xx}(0) |\psi(0)|^{p+1} \\\
&=  \int_{-\infty}^{+\infty} a_{xxx} (|\psi|^2)_x+ 2  \int_{-\infty}^{+\infty} a_{xx} |\psi_x|^2 - 2 a_{xx}(0) |\psi(0)|^{p+1} 
\end{align*}
Integrating by parts in the first term leaves no boundary terms since $a_{xxx}(0)=0$,
\begin{equation}
\label{E:1-158}
\text{II}=  -\int_{-\infty}^{+\infty} a_{xxxx}|\psi|^2 + 2 \int_{-\infty}^{+\infty} a_{xx} |\psi_x|^2 - 2 a_{xx}(0) |\psi(0)|^{p+1}
\end{equation}
Inserting \eqref{E:1-157} and \eqref{E:1-158} into \eqref{E:1-156} gives \eqref{E:local-virial}.

In the case $p=3$, there is an additional symmetry, \emph{pseudoconformal transform}, which is 
$$\psi(x,t) \text{ solves \eqref{E:1dcc}} \qquad \mapsto \qquad \tilde \psi(x,t) = \frac{e^{ix^2/4t}}{t^{1/2}}\psi( \frac{x}{t}, - \frac{1}{t}) \text{ solves \eqref{E:1dcc}.}$$  
Indeed, this is verified in two steps.  First, direct computation shows that
$$i\partial_t \tilde \psi + \partial_x^2 \tilde \psi = \frac{e^{-ix^2/4t}}{t^{5/2}}(i\partial_t \psi + \partial_x^2\psi) \left( \frac{x}{t}, - \frac{1}{t} \right) \quad \text{for }x\neq 0$$
Hence $\tilde \psi$ satisfies $i\partial_t\tilde\psi + \partial_x^2 \tilde \psi =0$ for $x\neq 0$ provided $\psi$ satisfies $i\partial_t\psi + \partial_x^2 \psi=0$ for $x \neq 0$.  Moreover, since
$$\tilde\psi_x(0\pm,t) = \frac{1}{t^{3/2}} \psi_x\left(0\pm, - \frac{1}{t} \right)$$
it follows that
$\tilde \psi$ satisfies the jump conditions \eqref{E:jump} if and only $\psi$ satisfies them.

Applying the pseudoconformal transformation to the ground state $\psi(x,t) = e^{it}\varphi_0(x)$, we obtain the solution
$$S_1(x,t) = \frac{e^{-i/t} e^{ix^2/4t}}{t^{1/2}} \varphi_0\left( \frac{x}{t} \right) \quad \text{for }t>0$$
Applying scaling gives
$$S_2(x,t) = \frac{e^{-i/\lambda^2 t} e^{ix^2/4t}}{(\lambda t)^{1/2}} \varphi_0\left( \frac{x}{\lambda t} \right) \quad \text{for }t>0$$
Using the time reversal symmetry ($\psi(x,t)$ solves \eqref{E:1dcc} implies that $\overline{\psi(x,-t)}$ solves \eqref{E:1dcc}) yields the solution
$$S_3(x,t) = \frac{e^{-i/\lambda^2 t} e^{ix^2/4t}}{(-\lambda t)^{1/2}} \varphi_0\left( -\frac{x}{\lambda t} \right) \quad \text{for }t<0$$
Replacing $t$ by $t-T_*$, for some $T_*>0$ gives the solution \eqref{E:1-148} quoted in the introduction.

\section{Sharp Gagliardo-Nirenberg inequality and applications}

\label{S:sharpGN}

In this section, we prove Prop. \ref{P:sharpGN}, Theorems \ref{T:global-blowup} and \ref{T:global-blowup-critical} on the sharp Gagliardo-Nirenberg inequality and its application to prove sharp criteria for global well-posedness and blow-up in both the $L^2$ supercritical and $L^2$ critical cases.

\begin{proof}[Proof of Prop. \ref{P:sharpGN}]
\begin{align*}
|\psi(0)|^2 &= \frac12 \int_{-\infty}^0 \frac{d}{dx} |\psi(x)|^2 \, dx - \frac12 \int_0^{+\infty} \frac{d}{dx} |\psi(x)|^2 \, dx \\
&= -\Re \int_{-\infty}^{+\infty} (\sgn x) \overline{\psi(x)} \psi'(x) \, dx
\end{align*}
The inequality then follows by Cauchy-Schwarz.  The fact that any $\psi$ of the form $\psi(x) =e^{i\theta} \alpha \varphi_0(\beta x)$ yields equality follows by direct computation.   Indeed, for any $p>1$,
$$\frac{ \|\varphi_0 \|_{L^2}^{1/2} \| (\varphi_0)_x \|_{L^2}^{1/2}}{\| \varphi_0 \|_{L^\infty}} = \int_{-\infty}^{+\infty} e^{-2|x|} \, dx = 1$$

 Now suppose that $|\psi(0)|^2 = \|\psi\|_{L^2} \|\psi'\|_{L^2}$.  Then $\psi$ is a minimizer (over $H^1$) of the functional
$$I(u) = \frac{ \|u\|_{L^2}^2 \|u'\|_{L^2}^2 }{ |u(0)|^4}$$
Hence $\psi$ solves the Euler-Lagrange equation
$$ 0 = \psi - \frac{\| \psi \|_{L^2}^2}{\|\psi'\|_{L^2}^2} \psi'' - \frac{2 \|\psi\|_{L^2}^2 }{|\psi(0)|^4} \delta |\psi|^2 \psi $$
Let $\tilde\psi(x) = \alpha^{-1}\psi(\beta^{-1} x)$ for some $\alpha>0$, $\beta>0$.  Then $\tilde\psi(x)$ solves 
$$ 0 = \tilde \psi - \frac{\beta^2 \|\psi\|_{L^2}^2}{ \|\psi'\|_{L^2}^2} \tilde \psi'' - \frac{2 \alpha^2\beta \|\psi \|_{L^2}^2 }{ |\psi(0)|^4} \delta |\tilde \psi|^2 \tilde \psi$$
Taking $\beta = \frac{\|\psi'\|_{L^2}}{\|\psi\|_{L^2}}$ and $\alpha = \frac{|\psi(0)|^2}{\sqrt{2} \|\psi\|_{L^2}^{1/2} \|\psi'\|_{L^2}^{1/2}}$ we obtain
$$0 = \tilde \psi - \tilde \psi'' - \delta |\tilde \psi|^2 \tilde \psi$$
The unique solution of this equation is $\tilde \psi(x) = e^{i\theta} \varphi_0(x)$. 
\end{proof}

We remark that the above proof is much more direct that its counterpart (Weinstein \cite{MR691044}) for \eqref{E:1-124}, the solitary wave profile for standard NLS, which relies on concentration compactness to construct a minimizer to a variational problem.  An overview of this result is included in Tao \cite[Apx. B]{MR2233925}.

\begin{proof}[Proof of Theorem \ref{T:global-blowup}]
Recall $\sigma_c = \frac12 - \frac1{p-1}$.  By direct computation, we find that
\begin{equation}
\label{E:1-161}
\frac{E(\varphi_0)}{\|(\varphi_0)_x\|_{L^2}^2} = \frac{p-3}{2(p+1)}  = \frac12 - \frac{2}{p+1}
\end{equation}
Moreover,
$$M(\psi)^{\frac{1-\sigma_c}{\sigma_c}} E(\psi) = \frac12 \|\psi \|_{L^2}^{\frac{2(1-\sigma_c)}{\sigma_c}} \|\psi_x \|_{L^2}^2 - \frac{1}{p+1} \|\psi\|_{L^2}^{\frac{2(1-\sigma_c)}{\sigma_c}} |\psi(0)|^{p+1}$$
Applying \eqref{E:1-126}
$$ M(\psi)^{\frac{1-\sigma_c}{\sigma_c}} E(\psi)\geq  \frac12 \|\psi \|_{L^2}^{\frac{2(1-\sigma_c)}{\sigma_c}} \|\psi_x \|_{L^2}^2 - \frac{1}{p+1} \|\psi\|_{L^2}^{\frac{p+1}{2}+\frac{2(1-\sigma_c)}{\sigma_c}} \|\psi_x\|_{L^2}^{\frac{p+1}{2}}$$
Using that $\frac{1-\sigma_c}{\sigma_c} = \frac{p+1}{p-3}$ and $\frac{p+1}{2} + \frac{2(1-\sigma_c)}{\sigma_c} = \frac{(p+1)^2}{2(p-3)}$, we can reexpress the right side to obtain
$$  M(\psi)^{\frac{1-\sigma_c}{\sigma_c}} E(\psi) \geq \frac12 \rho(t)^2 - \frac{1}{p+1} \rho(t)^{\frac{p+1}{2}}$$
where
$$\rho(t) = \| \psi\|_{L^2}^{\frac{1-\sigma_c}{\sigma_c}} \|\psi_x(t)\|_{L^2}$$
Dividing by (see \eqref{E:1-161})
\begin{equation}
\label{E:1-160}
M(\varphi_0)^{\frac{1-\sigma_c}{\sigma_c}} E(\varphi_0) = \| \varphi_0 \|_{L^2}^{\frac{2(1-\sigma_c)}{\sigma_c}} \| (\varphi_0)_x \|_{L^2}^2 \frac{p-3}{2(p+1)}
\end{equation}
we obtain
\begin{equation}
\label{E:1-155}
\frac{M(\psi)^{\frac{1-\sigma_c}{\sigma_c}} E(\psi)}{M(\varphi_0)^{\frac{1-\sigma_c}{\sigma_c}} E(\varphi_0)}  \geq f(\eta(t)) \defeq \frac{2(p+1)}{p-3} \left( \frac12 \eta(t)^2 - \frac{2}{p+1} \eta(t)^{\frac{p+1}{2}} \right)
\end{equation}
The function $f(\eta)$ has a maximum at $\eta=1$ with maximum value $f(1)=1$ (see Figure \ref{F:trapping}).  For each $y <1$, the equation $f(\eta)=y$ has two roots $\eta_-<1<\eta_+$.  Now take $\eta_-<1<\eta_+$ to be the two roots associated to $y=\frac{M(\psi)^{\frac{1-\sigma_c}{\sigma_c}} E(\psi)}{M(\varphi_0)^{\frac{1-\sigma_c}{\sigma_c}} E(\varphi_0)}$.  Since $\eta(t)$ is continuous, we obtain that either $\eta(t)\leq \eta_-$ for all $t$ (corresponding to case (1) in Theorem \ref{T:global-blowup}) or $\eta(t)\geq \eta_+$ for all $t$ (corresponding to case (2)).

We have established that in case (2), we have $\eta(t) \geq \eta_+>1$ on the whole maximal time interval $(T_-,T_+)$ of existence, and it remains to show that $|T_\pm|<\infty$, i.e. that $\psi(t)$ blows-up in finite negative and positive time.  In the local virial identity \eqref{E:local-virial}, we require $a(0)=a_x(0)=a_{xxx}(0)=0$.  If we in addition design $a(x)$ to satisfy $0\leq a(x) \leq C\epsilon^{-2}$, $a_{xx}(x) \leq 2$ and $|a_{xxxx}(x)|\leq C\epsilon^2$ for all $x\in \mathbb{R}$, and moreover $a_{xx}(0)=2$, then we have
$$\partial_t^2 \int a |\psi|^2 \leq 8 \|\psi_x\|_{L^2}^2 - 4|\psi(0)|^{p+1} + C\epsilon^2 M(\psi)$$
A weight function $a(x)$ meeting these conditions is given in the proof of Theorem \ref{T:global-blowup-critical}.  Multiply by $M(\psi)^{\frac{1-\sigma_c}{\sigma_c}}$ and divide by \eqref{E:1-160} to obtain
\begin{equation}
\label{E:1-162}
\begin{aligned}
\indentalign \frac{ M(\psi)^{\frac{1-\sigma_c}{\sigma_c}}}{ M(\varphi_0)^{\frac{1-\sigma_c}{\sigma_c}} E(\varphi_0)} \partial_t^2 \int a  |\psi|^2 \, dx \\
&\leq  4(p+1) \left( \frac{ M(\psi)^{\frac{1-\sigma_c}{\sigma_c}} E(\psi)}{ M(\varphi_0)^{\frac{1-\sigma_c}{\sigma_c}} E(\varphi_0)} - \eta(t)^2 \right) + \frac{C M(\psi)^{1/\sigma_c}}{M(\varphi_0)^{\frac{1-\sigma_c}{\sigma_c}}E(\varphi_0)} \epsilon^2
\end{aligned}
\end{equation}
Since $\frac{ M(\psi)^{\frac{1-\sigma_c}{\sigma_c}} E(\psi)}{ M(\varphi_0)^{\frac{1-\sigma_c}{\sigma_c}} E(\varphi_0)}<1$ and $\eta(t)\geq \eta_+>1$, we obtain that 
$$\delta \defeq - \frac{ M(\psi)^{\frac{1-\sigma_c}{\sigma_c}} E(\psi)}{ M(\varphi_0)^{\frac{1-\sigma_c}{\sigma_c}} E(\varphi_0)}+\eta_+^2>0$$
By \eqref{E:1-162}, it follows that
$$\partial_t^2 \int |x|^2 |\psi|^2 \, dx \leq -C_1\delta+ C_2\epsilon^2$$
for constants $C_j>0$.  Now take $\epsilon>0$ sufficiently small so that the right side is still bounded by a strictly negative number.  This forces $\int a |\psi|^2 \, dx$ to become zero in finite negative time $-\infty<\hat T_-<0$ and in finite positive time $0<\hat T_+<\infty$.  Since $a(x)\geq 0$, the maximal time interval of existence $(T_-,T_+)$ must be contained in $(\hat T_-,\hat T_+)$, and in particular $|T_\pm|<\infty$.  Note that since $a(x)$ is bounded, this argument does not require the assumption of finite variance ($\int x^2 |\psi|^2 \, dx <\infty$)

\begin{figure}
\includegraphics[scale=0.6]{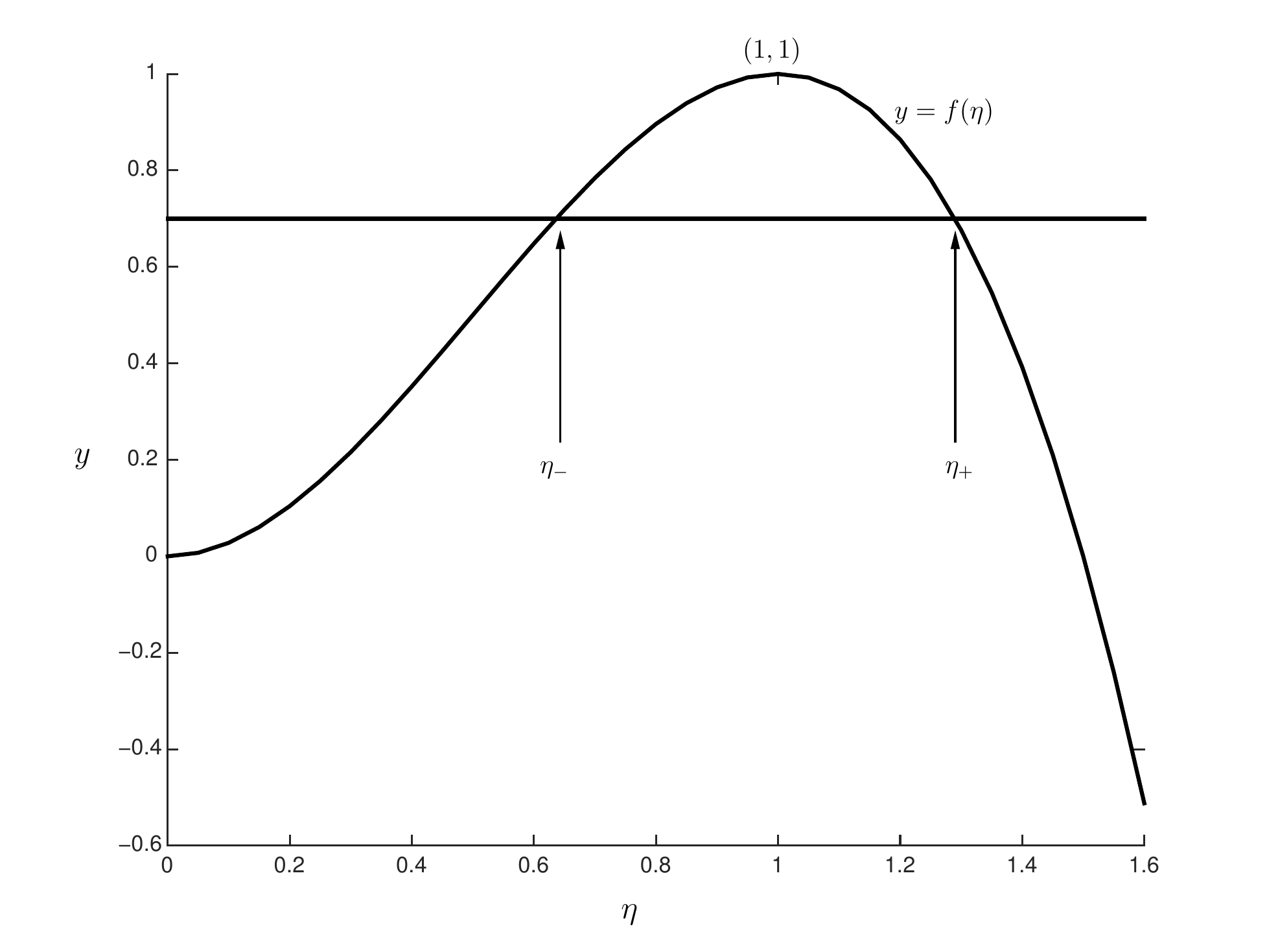}
\caption{\label{F:trapping} A plot of $y=f(\eta)$ versus $\eta$, where $f(\eta)$ is defined in \eqref{E:1-155}.  The function $f(\eta)$ has a maximum value of $1$ at $\eta=1$.  For each $0\leq y<1$, there are two roots $0\leq \eta_-<1<\eta_+$ of the equation $f(\eta)=y$, for each $y<0$, there is one root $\eta_+>1$.}
\end{figure}

\end{proof}

\begin{proof}[Proof of Theorem \ref{T:global-blowup-critical}]
Suppose that $M(\psi)<M(\varphi_0)=2$.  Note that
$$2E(\psi) = \|\psi'\|_{L^2}^2 - \tfrac12 |\psi(0)|^4 \geq \|\psi'\|_{L^2}^2 - \tfrac12 \|\psi\|_{L^2}^2 \|\psi ' \|_{L^2}^2 = \|\psi'\|_{L^2} ^2(1 - \tfrac12 M(\psi))$$
where in the last step we applied \eqref{E:1-146} from Prop. \ref{P:sharpGN}.

Now suppose that $E(\psi)<0$ (which can only happen if $M(\psi)>M(\varphi_0)=2$).  We apply the local virial identity \eqref{E:local-virial}.  In that equality, if $a(x)$ is chosen so that $a''(0)=2$ and $a''(x) \leq 2$ for all $x$, then
$$\partial_t^2 \int a(x) |\psi (x,t)|^2 \, dx \leq  16 E(\psi) - \int a''''(x) |\psi(x,t)|^2 \, dx$$
Let $b(x)$ be any smooth compactly supported function such that $b(x)=2$ for $0 \leq  x \leq 1$, $b(x) \leq 2$ for all $x\geq 1$, and  $\int_0^{+\infty} b(x) \, dx =0$.   Since $b$ is compactly supported and $\int_0^{+\infty} b(y) \, dy =0$, we have that the first integral $d(x) = \int_0^x b(y) \, dy$ is also compactly supported.  Finally, let $e(x) = \int_0^{|x|} d(y) \, dy$.  Then $e(x)$ is bounded, even, and $e(x) = x^2$ for $-1 \leq x \leq 1$.   Now set $a(x) = \epsilon^{-2} e(\epsilon x)$.   Then $a(x) = x^2$ for $-\epsilon^{-1} \leq x \leq \epsilon^{-1}$ so that $a(0)=0$, $a'(0)=0$, and $a''(0)=2$ and $a'''(0)=0$.  Moreover $a''(x) = e''(\epsilon x) = b(\epsilon x) \leq 2$.  Also note that $a''''(x) = \epsilon^2e''''(\epsilon x) = \epsilon^2 b''(\epsilon x) \leq C \epsilon^2$ (since $b$ is smooth, compactly supported).   Hence
$$ \partial_t^2 \int a(x) |\psi (x,t)|^2 \, dx \leq  16 E(\psi) + C\epsilon^2 M(\psi) $$
Take $\epsilon>0$ sufficiently small so that the right side is $<0$.  Then $\int a(x) |\psi(x,t)|^2 \, dx$ becomes negative at some finite time $T>0$, implying blow-up prior to $T$.  Note that since $a(x)$ is bounded, this does not require that $\psi$ belong to a weighted $L^2$ space.  
\end{proof}

\section{$L^2$ critical mass concentration and minimal mass blow-up}

\label{S:L2}

Before we proceed, let us remark that the following standard theorem in real analysis has an especially simple proof in 1D.

\begin{proposition}[Rellich-Kondrachov compactness]
\label{P:RK}
Suppose that $v_n\rightharpoonup V$ (weakly) in $H_x^1$.  Then $v_n\to V$ pointwise and for any $A>0$,
\begin{equation}
\label{E:1-145}
\int_{|x|\leq A} |v_n(x)|^2 \, dx\to \int_{|x|\leq A} |V(x)|^2 \, dx
\end{equation}
\end{proposition}
\begin{proof}
By translation, it suffices to show that $v_n(0) \to V(0)$.   By the fundamental theorem of calculus,
\begin{align*}
v_n(0)V(0) &= -\int_0^{+\infty} \frac{d}{dx} [v_n(x) V(x) ] \, dx \\
&= - \int_0^{+\infty} v_n'(x) V(x) \, dx - \int_0^{+\infty} v_n(x) V'(x) \, dx
\end{align*}
By the definition of weak convergence and the fundamental theorem of calculus again
$$v_n(0)V(0) \to - 2\int_0^{+\infty} V'(x)V(x) \, dx = V(0)^2$$
If $V(0) \neq 0$, then we conclude that $v_n(0) \to V(0)$.  If $V(0)=0$, then we can replace $v_n(x)$ by $v_n(x)+e^{-x^2}$ and $V(x)$ by $V(x) + e^{-x^2}$.    This completes the  proof of the pointwise convergence.  Since a weakly convergent sequence is bounded, we have by \eqref{E:1-126}
$$|v_n(x)| \leq \|v_n(x)\|_{L^2}^{1/2} \|v_n'(x)\|_{L^2}^{1/2}  \leq C$$
Thus \eqref{E:1-145} follows from the pointwise convergence and the dominated convergence theorem. 
\end{proof}

The following lemma is modeled on Hmidi \& Keraani \cite[Theorem 1.1]{hmidi2005blowup}, although the proof is much simpler since concentration compactness (profile decomposition) is not needed.

\begin{lemma}
\label{L:3}
Suppose that $\{v_{n}\}^{\infty}_{n=1} \subset H^{1}(\mathbb{R})$ is a bounded sequence such that
$$m \leq \lim_{n\to\infty} \left( \frac{|v_{n}(0)|}{|\varphi_{0}(0)|} \right)^{4} \,, \qquad 
\lim_{n\to \infty} \left( \frac{\|v_{n}'\|_{L^{2}}}{\|\varphi_{0}'\|_{L^{2}}} \right)^{2} \leq M \, . $$
Then there exists a subsequence (still labeled $v_{n}$) such that
$$v_{n} \rightharpoonup V \text{\ in\ } H^{1}$$
with 
\begin{equation}
\label{E:1-146}
\frac mM \leq \left( \frac{\|V\|_{L^{2}}}{\|\varphi_{0}\|_{L^{2}}} \right)^{2}
\end{equation}
\end{lemma}
\begin{proof}
From the boundedness of $\{v_{n}\}_{n=1}^{\infty}$, we can pass to a subsequence (still labeled $v_{n}$) such that $v_{n} \rightharpoonup V$ for some $V \in H^{1}$. 
We have $v_{n}(0) \to V(0)$ by Prop. \ref{P:RK}. Hence,
$$m \leq \lim_{n\to \infty} \left( \frac{ |v_n(0)|}{|\varphi_0(0)|} \right)^4 = \left( \frac{|V(0)|}{|\varphi_0(0)|} \right)^4$$
and
$$ \left( \frac{\|V'\|_{L^2}}{\|\varphi_0'\|_{L^2}} \right)^2 \leq \lim_{n\to \infty} \left( \frac{\|v_n'\|_{L^2}}{\|\varphi_0'\|_{L^2}} \right)^2 \leq M$$
Applying \eqref{E:1-126} to $V$, we have
$$\left( \frac{|V(0)|}{|\varphi_{0}(0)|}\right)^{4} \leq \left(\frac{\|V\|_{L^{2}}}{\|\varphi_{0}\|_{L^{2}}} \right)^{2} \left(\frac{\|V'\|_{L^{2}}}{\|\varphi_{0}'\|_{L^{2}}} \right)^{2}$$
Combining the inequalities above, we get \eqref{E:1-146}.
\end{proof}

We can now apply the above to prove Theorem \ref{T:mass-conc}.
\begin{proof}[Proof of Theorem \ref{T:mass-conc}]
Let $\rho(t) = \| \varphi_0' \|_{L^2}/\|\psi'(t)\|_{L^2}$ and $v(x,t) = \rho(t)^{1/2}\psi(\rho(t)x,t)$ so that  $\|v'(t)\|_{L^2} = \|\varphi_0'\|_{L^2}$.  Moreover, 
$$E(v(t)) = \rho(t)^2 E(\psi(t)) = \rho(t)^2E(\psi_0) \to 0 \text{ as }t\nearrow T_*$$
Hence
$$\lim_{t\nearrow T_*} |v(0,t)|^4 = 2\lim_{t\nearrow T_*} \|v'(t)\|_{L^2}^2 = 2 \|\varphi_0'\|_{L^2}^2 = |\varphi_0(0)|^4$$
Let 
$$H(t) = \int_{|x|\leq \mu(t) \|\psi'(t)\|_{L^2}^{-1}} |\psi(x,t)|^2 \, dx = \int_{|x| \leq \mu(t)} |v(x,t)|^2 \, dx$$
Take $t_n\to T_*$ so that $\lim_{n\to \infty} H(t_n) = \liminf_{t\nearrow T_*} H(t)$.  Apply Lemma \ref{L:3}  to $v_n = v(t_n)$, with $m=M=1$, passing to a sequence obtain $v_n \rightharpoonup V$ such that $\|V\|_{L^2}/\|\varphi_0\|_{L^2} \geq 1$.  For each $A>0$, we have
$$\lim_{n\to \infty} H(t_n) \geq \lim_{n\to \infty} \int_{|x|\leq A} |v(x,t_n)|^2 \, dx = \int_{|x|\leq A} |V(x)|^2 \, dx$$
by Prop. \ref{P:RK}.  Since $A>0$ is arbitrary,
$$\lim_{n\to \infty} H(t_n) \geq  \int |V(x)|^2 \, dx \geq \|\varphi_0\|_{L^2}^2 $$
\end{proof}

Now we proceed to prove Theorem \ref{T:minimal}, stating that a minimal mass blow-up solution is necessarily a pseudoconformal transformation of the ground state.   The corresponding result for standard NLS \eqref{E:1-121} is due to Merle \cite{merle1993determination} and later simplified by Hmidi \& Keraani \cite[Theorem 2.4]{hmidi2005blowup}.  Our argument is modeled on Hmidi-Keraani \cite{hmidi2005blowup}.

\begin{proof}[Proof of Theorem \ref{T:minimal}]
Let $\alpha(s) \in C_c^\infty([0,+\infty))$ such that $0 \leq \alpha(s) \leq 1$ for all $s$, $\alpha(s)=1$ for $s\leq 1$ and $\alpha(s)=0$ for $s\geq 2$.  Let $\beta\in C_c^\infty(\mathbb{R})$ be given by $\beta(x) = |x|^2 \alpha(|x|^2)^2$.  A computation shows that there exists $C>0$ such that $|\beta'(x)|^2 \leq C\beta(x)$.  Let $\beta_p(x) = p^2\beta(x/p)$ and
$$g_p(t) = \int \beta_p(x) |\psi(x,t)|^2 \, dx$$
From \eqref{E:1dcc}, we compute
\begin{equation}
\label{E:1-19}
g_p'(t) = 4 \Im \int \beta_p'(x) \; \psi'(x,t) \, \overline{\psi(x,t)} \, dx
\end{equation}
By \eqref{E:1-126} in Prop. \ref{P:sharpGN}, if $\|v\|_{L^2}=\|\varphi_0\|_{L^2}$, then $E(v) \geq 0$.  Hence, for each $s\in \mathbb{R}$, $E(e^{is\beta_p}\psi) \geq 0$.  Since $E(e^{is\beta_p}\psi)$ is a quadratic polynomial in $s$,
$$E(e^{is\beta_p}\psi) = \frac12 s^2 \int (\beta_p')^2 |\psi|^2 \, dx + s \Im \int \beta_p' \, \psi' \, \bar \psi \, dx + E(\psi) $$
the discriminant is nonpositive, i.e.  
$$\left( \Im \int \beta_p' \, \psi' \, \bar \psi \, dx \right)^2 \leq 2 E(\psi) \int (\beta_p')^2 |\psi|^2 \, dx $$
Plugging in \eqref{E:1-19} and using that $(\beta_p')^2 \leq C \beta_p$ we get
$$|g_p'(t)| \leq C(\psi_0) g_p(t)^{1/2}$$
Integrating from time $t$ to time $T$, we obtain
\begin{equation}
\label{E:1-20}
|g_p(T)^{1/2} - g_p(t)^{1/2}| \leq \frac12 C(\psi_0)(T-t)
\end{equation}
We would like to take $T=T_*$ with $g_p(T_*)=0$ in this inequality to conlude $g_p(t) \leq \tilde C(\psi_0) (T_*-t)^2$ for all $0\leq t <T_*$, but instead we must take $T\nearrow T_*$ in an appropriate limiting sense.  Let $t_n \nearrow T_*$ be any sequence of times approaching $T_*$.  By Theorem \ref{T:mass-conc} and the fact that $\|\psi(t)\|_{L^2} = \|\varphi_0\|_{L^2}$, it follows that for each $\epsilon>0$,
$$ \lim_{n\to \infty} \int_{|x|\leq \epsilon} |\psi(x,t_n)|^2 \, dx = \|\varphi_0\|_{L^2}^2$$
Hence for every $\epsilon>0$,
\begin{equation}
\label{E:1-150}
\lim_{n\to \infty} \int_{|x|> \epsilon} |\psi(x,t_n)|^2 \,dx =0
\end{equation}
Since $\beta_p(x) = |x|^2$ for $|x|\leq p$ and $\beta_p(x) \leq 4p^2$ for all $x\in \mathbb{R}$, if $0<\epsilon<p$, we have
$$g_p(t_n) \leq 4p^2 \int_{|x|>\epsilon} |\psi(x,t_n)|^2 \, dx + \epsilon^2 \int_{|x|<\epsilon} |\psi(x,t_n)|^2 \, dx$$
Sending $n \to \infty$, by \eqref{E:1-150} we have
$$\lim_{n\to \infty} g_p(t_n) \leq \epsilon^2 \| \varphi_0\|_{L^2}^2$$
Since $\epsilon>0$ is arbitary, we have (for fixed $p>0$) that $\lim_{n\to \infty} g_p(t_n) =0$.  Plugging into \eqref{E:1-20} with $T=t_{n}$ and sending $n\to\infty$, we obtain
$$g_{p}(t) \le C(\psi_{0})^{2} (T_{*} - t)^{2}$$
Sending $p\to\infty$, we conclude that $\psi\in \Sigma$ and
\begin{equation}
\label{E:1-21}
g(t) \leq C(\psi_{0})^{2} (T_{*}-t)^{2}
\end{equation}
where
$$g(t) \defeq \int |x|^{2} |\psi(x,t)|^{2}\, dx$$
Here, $\Sigma$ is the space of functions for which the norm $\| \psi \|_{\Sigma} \defeq ( \|\psi \|_{H_x^1} + \|x\psi\|_{L^2})^{1/2}$ is finite, i.e. $H_x^1$ functions of finite variance.

The virial identity is $$g''(t) = 16 E(\psi_{0})$$
Integrating twice,
$$g(t) = g(0) + tg'(0) + 8t^{2}E(\psi_{0}) = 8t^{2} E(e^{i|x|^{2}/(4t)}\psi_{0})$$
Plugging into \eqref{E:1-21},
$$ 8t^{2} E(e^{i|x|^{2}/(4t)}\psi_{0}) \leq C(\psi_{0})^{2}(T_{*}-t)^{2} $$
Sending $t \nearrow T_{*}$, we obtain $E(e^{i|x|^{2}/(4t)}\psi_{0})=0$.  Hence
$$ 2\| (e^{i|x|^2/4T_*} \psi_0(x))' \|_{L^2}^2 = \left| (e^{i|x|^2/4T_*} \psi_0(x))\Big|_{x=0} \right|^4$$
Since $\| e^{i|x|^2/4T_*} \psi_0(x) \|_{L^2}^2 = \|\varphi_0\|_{L^2} = 2$, this implies that $e^{i|x|^2/4T_*} \psi_0(x)$ gives equality in \eqref{E:1-126}.  Prop. \ref{P:sharpGN} then implies that there exists $\theta\in \mathbb{R}$, $\alpha>0$ and $\beta>0$ such that
$$e^{i|x|^2/4T_*}\psi_0(x) = e^{i\theta} \alpha \varphi_0(\beta x)$$
Taking the $L^2$ norm of this equation and using that $\|\psi_0\|_{L^2} = \|\varphi_0\|_{L^2}$ gives that $\alpha = \beta^{1/2}$.  Hence \eqref{E:1-147} holds.

\end{proof}

\section{$L^2$ critical near minimal mass blow-up}
\label{S:near-minimal}

In this section, we state and prove the supporting lemmas for Theorem \ref{T:near-minimal}, which states that near minimal mass blow-up solutions are close to modulations of the ground state $\varphi_0$.

\begin{lemma}
\label{L:2}
For each $\epsilon>0$ there exists $\delta>0$ such that the following holds.
If $v\in H^1$ is such that 
$$| \|v'\|_{L^2} - \|\varphi_0'\|_{L^2} | \leq \delta$$
$$| \|v\|_{L^2} - \|\varphi_0\|_{L^2} |\leq \delta$$
$$||v(0)| - |\varphi_0(0)| | \leq \delta$$
then there exists $\theta\in \mathbb{R}$ such that 
$$\| v(x) - e^{i\theta} \varphi_0(x) \|_{H^1} \leq \epsilon$$
\end{lemma}
\begin{proof}
The assertion is equivalent to the following statement about sequences:  Suppose that $v_n$ is a sequence in $H^1$ such that $\|v_n'\|_{L^2} \to \|\varphi_0'\|_{L^2}$, $\|v_n\|_{L^2}\to \|\varphi_0\|_{L^2}$, $|v_n(0)| \to |\varphi_0(0)|$.  Then there exists $\theta\in \mathbb{R}$ and a subsequence (still labeled $v_n$) such that $v_n \to e^{i\theta}\varphi_0$ in $H^1$.  To prove this, pass to a subsequence such that $v_n \rightharpoonup \psi$ for some $\psi\in H^1$.  We will show that there exists $\theta\in \mathbb{R}$ such that $\psi = e^{i\theta}\varphi_0$.  By Prop. \ref{P:RK}, $v_n(0)\to \psi(0)$.  Furthermore, $v_n \rightharpoonup \psi$ implies $\|\psi\|_{L^2} \leq \lim_{n\to \infty} \|v_n\|_{L^2}$ and $\|\psi'\|_{L^2} \leq \lim_{n\to \infty} \|v_n'\|_{L^2}$. Hence
$$I(\psi) \leq \lim_{n\to \infty} I(v_n) = I(\varphi_0)$$
Since $\varphi_0$ is a minimizer of $I$, $\psi$ is also a minimizer, $I(\psi) = I(\varphi_0)$, and hence $\|\psi\|_{L^2} = \lim_{n\to \infty} \|v_n\| = \|\varphi_0\|_{L^2}$ and $\|\psi'\|_{L^2} = \lim_{n\to \infty} \|v_n'\|_{L^2}= \|\varphi_0'\|_{L^2}$.  This together with the fact that $v_n \rightharpoonup \psi$ implies that $v_n \to \psi$ (strongly)  in $H^1$.  

By the uniqueness in Prop. \ref{P:sharpGN}, there exists $\theta\in \mathbb{R}$, $\alpha>0$, and $\beta>0$ such that $\psi(x) = e^{i\theta} \alpha \varphi_0(\beta x)$.  Since $\|\psi\|_{L^2} = \|\varphi_0\|_{L^2}$ and $\|\psi'\|_{L^2} = \|\varphi_0'\|_{L^2}$ it follows that $\alpha=1$ and $\beta=1$.  

\end{proof}

\begin{corollary}
\label{C:near-minimal}
For each $\epsilon>0$, there exists $\delta>0$ such that the following holds.  If $v\in H^1$ is such that $| \|v\|_{L^2} - \|\varphi_0\|_{L^2} | \leq \delta$ and $|E(v)| \leq \rho^{-2}\delta$, where $\rho = \frac{\|\varphi_0'\|_{L^2}}{\|v'\|_{L^2}}$, then there exists $\theta\in \mathbb{R}$ such that $\|e^{-i\theta} \rho^{1/2} v(\rho x) - \varphi_0(x) \|_{H^1} \leq \epsilon$.
\end{corollary}
\begin{proof}
Let $\tilde v(x) = \rho^{1/2} v(\rho x)$.  We will apply Lemma \ref{L:2} to $\tilde v$.  First,  the hypothesis $| \|v\|_{L^2} - \|\varphi_0\|_{L^2} | \leq \delta$ implies $| \|\tilde v\|_{L^2} - \|\varphi_0\|_{L^2} | \leq \delta$.  Second, $\|\tilde v'\|_{L^2} = \rho \|v'\|_{L^2} = \|\varphi_0'\|_{L^2}$.  Finally, the fact that $| |\tilde v(0)| - |\varphi_0(0)| | \leq \tilde \delta$ follow from:
\begin{align*}
| |\tilde v(0)|^4 - |\varphi_0(0)|^4 | &= | 2 \rho^2 \|v'\|_{L^2}^2 - 4\rho^{2} E(v) - |\varphi_0(0)|^4| \\
&= | 2\|\varphi_0'\|_{L^2}^2 - |\varphi_0(0)|^4 - 4\rho^2 E(v)|\\
&= 4\rho^2 |E(v)| \\
&\leq 4\delta 
\end{align*}
\end{proof}

Theorem \ref{T:near-minimal} follows from Corollary \ref{C:near-minimal} by taking $v = \psi(t)$ for each $t$ for which $|E(\psi)| \leq \rho(t)^{-2} \delta$.  This inequality is valid for some interval $(T_*-\delta,T_*)$, since $\rho(t) \to 0$ as $t \nearrow T_*$.

\def\cprime{$'$} \def\cprime{$'$}
\providecommand{\bysame}{\leavevmode\hbox to3em{\hrulefill}\thinspace}
\providecommand{\MR}{\relax\ifhmode\unskip\space\fi MR }
\providecommand{\MRhref}[2]{%
  \href{http://www.ams.org/mathscinet-getitem?mr=#1}{#2}
}
\providecommand{\href}[2]{#2}

\end{document}